\documentclass[11pt]{amsart}

\usepackage[utf8]{inputenc}
\usepackage[english]{babel}
\usepackage{geometry}   
\usepackage{setspace}
\usepackage{graphicx}
\usepackage{amssymb, amsthm}
\usepackage{tikz}
\usepackage{tikz-3dplot,pgfplots}
\usepackage{color}
\usepackage{subfiles}			

 \newtheorem{theorem}{Theorem}[section]

  \newtheorem{lemma}[theorem]{Lemma}

   \theoremstyle{definition}
\newtheorem{definition}[theorem]{Definition}
\newtheorem{remark}[theorem]{Remark}

\newcommand{\Z}{\ensuremath{{\mathbb{Z}}}}
\newcommand{\R}{\ensuremath{{\mathbb{R}}}}
  \newcommand{\C}{\ensuremath{{\mathbb{C}}}}

  \newcommand{\G}{\Gamma}
    \newcommand{\AG}{\ensuremath{{A_{\Gamma}}}}
        \newcommand{\WG}{\ensuremath{{W_{\Gamma}}}}
   \newcommand{\CG}{\ensuremath{{\mathcal{C}_{\Gamma}}}}  
     \newcommand{\DG}{\ensuremath{{\mathcal{D}_{\Gamma}}}}  
          
            \newcommand{\DT}{\ensuremath{{\mathcal{D}_{T}}}}  
     \newcommand{\Scl}{\ensuremath{{\mathcal{S}^{\Delta}}}} 
\newdimen\Rad
\tikzset{vertex/.style={circle, draw, fill=black!50},inner sep=0pt, minimum width=4pt}

\graphicspath{./figures}

\title{Artin groups of infinite type: trivial centers and acylindical hyperbolicity}
\author{Ruth Charney and Rose Morris-Wright}

\thanks {R. Charney was partially supported by NSF grant DMS-1607616}

\begin{document}

\begin{abstract}  
While finite type Artin groups and right-angled Artin groups are well-understood, little is known about more general Artin groups.  In this paper we use the action of an infinite type Artin group $A_\G$ on a CAT(0) cube complex to prove that $A_\G$ has trivial center providing $\G$ is not the star of a single vertex, and is acylindrically hyperbolic providing $\G$ is not a join. 
\end{abstract}

\maketitle

\section{Introduction and Background}

	This paper concerns properties of Artin groups.  Artin groups are closely connected to Coxeter groups, so we begin by recalling the definitions of these two classes of groups.  Let $\G$ be a finite simplicial graph with vertices 
$S=\{s_1, \dots s_n\}$ and edges $e(s_i,s_j)$ labeled by integers $m_{i,j} \geq 2$.  The \emph{Coxeter group} associated to $\G$ is the group with presentation
$$\WG = \langle S \mid s_i^2=1, (s_is_j)^{m_{i,j}}=1\rangle $$
where the first relation holds for all $i$ and the second holds whenever $s_i,s_j$ are connected by an edge.
 For convenience, we often say $m_{i,j}=\infty$ if $s_i,s_j$ are not connected by an edge.
The \emph{Artin group} associated to $\G$ is the group with presentation
$$A=\langle s_1,\dots,s_n\ \mid\ 
\underbrace{s_is_js_i\dots}_{m_{ij}\text{ terms}}=\underbrace{s_js_is_j\dots}_{m_{ij}\text{ terms}}\ 
\textrm{ for all } i\neq j\rangle \,.$$
There is a natural surjection of $\AG \to \WG$ which takes each generator in $\AG$ to the generator of the same name in $\WG$.  An Artin group $\AG$ is said to be \emph{finite type} (respectively \emph{infinite type}) if the corresponding Coxeter group $\WG$ is finite (respectively infinite).  Artin groups first appear as ``extended Coxeter groups" in a paper of Tits \cite{Ti} and they are often called \emph{Artin-Tits groups}.

There is also a geometric description of these groups.  Coxeter groups can be realized as discrete subgroups of $GL(n,\R)$ with the generators (and their conjugates) acting as reflections.  Complexifying this action gives an action of $\WG$ on $\C^n$.  The set of regular points of this action (that is, points with trivial stabilizer) is the hyperplane complement
$$\mathcal H_\G = \C^n - \bigcup_r H_r$$
where $r$ ranges over all conjugates of the generators in $\WG$, and $H_r$ denotes the complex hyperplane fixed by the reflection $r$.  $\WG$ acts freely on this space, and 
in the case of a finite Coxeter group, $\mathcal H_\G / \WG$ has fundamental group $\AG$.  Moreover, it was shown by Deligne to be a $K(\AG,1)$-space \cite{Del}. 
The classical example is the case where $\WG$ is the symmetric group on $n$ letters acting on $\C^n$ by permuting the coordinates.  In this case, $\mathcal H_\G / \WG$ is the configuration space of $n$ (unordered) points in the complex plane, and its fundamental group, $\AG$, is well-known to be the braid group on $n$ strands.
Artin groups associated to infinite Coxeter groups also arise as fundamental groups of hyperplane complements (after restricting to an open cone in $\C^n$), but the question of whether this space is a  $K(\pi,1)$-space remains an open conjecture.  For a thorough discussion of this question see Paris' survey article \cite{Par2}.

Coxeter groups have been extensively studied using methods from algebra, combinatorics, geometry, and representation theory.  In general, these techniques apply to both finite and infinite Coxeter groups, and both classes are well-understood.  Finite type Artin groups are also well understood, thanks in large part, to a particularly nice combinatorial structure, known as a Garside structure, that gives rise to nice normal forms and effective approaches to many algebraic questions.
Certain infinite type Artin groups, such as right-angled Artin groups (all $m_{i,j}=2$ or $\infty$) and extra-large Artin groups (all $m_{i,j}\geq 4$), are also fairly well understood.
More general infinite type Artin groups, on the other hand, remain largely mysterious.  Here are some long-standing conjectures that remain open for a general Artin group of infinite type. 
\medskip

\noindent{\bf Conjectures:}  Let $\AG$ be an Artin group of infinite type.
\begin{enumerate}
\item $\AG$ has solvable word problem.
\item $\AG$ is torsion-free.
\item If $\AG$ is irreducible (does not split as a direct product of two Artin subgroups), then $\AG$ has trivial center. 
\item There exists a finite $K(\AG,1)$-space (i.e., $\AG$ is of type F).
\item  The complex hyperplane complement associated to $\WG$ is a $K(\AG,1)$-space.
\end{enumerate}
\medskip

More recently, there has been interest in understanding geometric properties of Artin groups.  For example, which of these groups are CAT(0) groups?  Which are acylindrically hyperbolic?  These properties involve actions of a group on spaces of non-positive or negative curvature and have been shown to have strong implications for the group itself.  Recent results by Calvez and Wiest \cite{Ca-Wi} show that irreducible finite type Artin groups (modulo their center) are acylindrically hyperbolic, but the question of whether Artin Groups are CAT(0) remains open even for braid groups. For some partial results, see \cite{Br-Mc, Ha-Th-Sc, Hae}.  (We remark that the only Artin groups that are hyperbolic are free groups since for any graph $\G$ containing an edge, $\AG$ can easily be shown to contain a $\Z^2$-subgroup.  Moreover, Behrstock, Drutu, and Mosher have shown that the only Artin groups that are relatively hyperbolic are those that decompose as direct products \cite{BeDrMo}.)

There is one other class of infinite type Artin groups for which we know that all of the conjectures listed above hold, namely the Artin groups of FC-type.  These were introduced by the first author and M.~Davis in \cite{Ch-Da1} where it was shown that a certain cubical complex $D_\G$, called the Deligne complex, is homotopy equivalent to the universal cover of the hyperplane complement for $\WG$.  If $D_\G$ is CAT(0), then it is contractible, so the $K(\pi,1)$-conjecture (Conjecture (5) above) holds for $\AG$.  This holds precisely when every clique $\Theta \subset \G$ generates a finite subgroup $W_\Theta$ (or equivalently, a finite type subgroup $A_\Theta$).  Artin groups satisfying this condition are called \emph{Artin groups of FC-type}.  (Here, ``FC" stands for Flag Complex, the condition on links in $D_\G$ required for the CAT(0) property to hold.)  It turns out that all of the above conjectures can be proved in the FC type case using the Deligne complex  (\cite{Al-Ch, Ch-Da1, God}).  Moreover, Chatterji and Martin \cite{Ch-Ma} have recently used the action of $\AG$ on $D_\G$ to show that FC-type Artin groups are acylindrically hyperbolic providing the defining graph has diameter at least 3. 

Variations on the Deligne complex were introduced by Godelle and Paris in \cite{Go-Pa1}.  In this paper we focus on one of those complexes, which we call the clique-cube complex. The clique-cube complex is CAT(0) for \emph{any} graph $\G$.  We use this complex to prove several new results for general Artin groups.

In Section 3, we show that the center of $\AG$ is trivial, providing $\G$ is not the star of a single vertex (Theorem \ref{thm: trivial center}). In Section 4, we prove that any Artin group $\AG$ is acyindrically hyperbolic, providing $\G$ does not decompose as the join of two subgraphs (Theorem \ref{thm: acyl hyp}).  In the final section, we give simple proofs of some prior results of Ellis-Sk\"oldberg \cite{El-Sk} and Godelle-Paris \cite{Go-Pa1} reducing several of the other conjectures to the case where $\G$ consists of a single clique.  

The authors would like to thank Anthony Genevois for helpful comments.  The first author would like to thank the Mathematics Institute at Warwick University for their hospitality during the completion of this paper.

\section{Introducing the Clique-Cube Complex} \label{Cube Complex}

	Let $\G$ be a finite simplicial graph with vertices $S=\{s_1, \dots s_n\}$ and edges $e(s_i,s_j)$ labeled by integers $m_{i,j} \geq 2$.  
Let $\AG$ denote the Artin group associated to $\G$.   By a theorem of van der Lek (\cite{vdL, Par1}),
for any subset $T \subset S$, the subgroup of $\AG$ generated by $T$ is itself an Artin group with defining graph
consisting of the full subgraph of $\G$ spanned by $T$.  We denote this subgroup by $A_T$ and by abuse of notation, frequently conflate $T$ with the subgraph generated by $T$.  When $T$ is empty, we define $A_\emptyset=\{1\}$. 

We say $T$ spans a \emph{clique} in $\G$ if any two elements of $T$ are joined by an edge in $\G$. 

\begin{definition}
Consider the set 
$$ \Scl = \{ T \subseteq S \mid \textrm{$T$ spans a clique in $\G$, or $T = \emptyset$}\}$$

 The \textit{clique-cube complex}, denoted $\CG,$ is the cube complex whose vertices are cosets $gA_T$, $T \in \Scl$, where $gA_T$ and $hA_{T'}$ are joined by an edge if and only if $gA_T \subset hA_{T'}$  and $T$ and $T'$ differ by a single generator. Note that in this case, we can always replace $h$ by $g$, that is $hA_{T'}=gA_{T'}$.   More generally, for any pair $gA_T \subset gA_{T'}$, the interval $[gA_T, gA_{T'}]$ spans a cube of dimension $|T' \backslash T|$.
\end{definition}

\begin{figure}
 
\begin{tikzpicture}[ scale=1.3]
  						 \draw 
    								\foreach \x in {-2,-1,2,3,4,5,6} {
            									(0,0)-- (\x*30:2\Rad) node[vertex,label={[shift={(0.35,0.1)}]$s^{\x}A_\emptyset$}]{}
        							}; 
        							
        							\draw (0,0)--(30:2\Rad) node[vertex,label={[shift={(0.35,0.1)}]$sA_\emptyset$}]{};
        							
        							\draw[blue, very thick] (0,0)-- (0cm:2\Rad) node[vertex,fill=blue,label={[shift={(0.3,0.1)}]$A_\emptyset$}]{}
        							(0:0) node[vertex,fill=blue,label={[shift={(-0.4,-0.6)}]$A_{\{s\}}$}]{};
        							
        							\draw
        							
        							(-70:1.85\Rad) node[label={$\dots$}]{}
        							(193:1.8\Rad) node[label={[rotate={-45}]$\dots$}]{}
        							;
\end{tikzpicture}
\caption{The complex $\CG$ where $\G$ is the graph consisting of a single vertex. The \textcolor{blue}{fundamental domain} of the action of $\AG$ on $\CG$ is the edge spanned by $A_{\{s\}}$ and $A_\emptyset$. }
\label{fig:Example 1}
\end{figure}
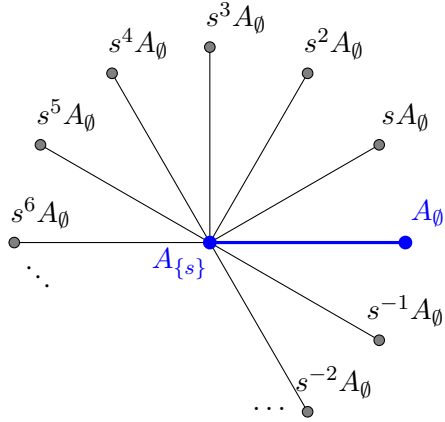

The construction of the clique-cube complex is similar to that of the Deligne complex introduced in \cite{Ch-Da1}. The Deligne complex is a cube complex defined in an identical fashion by replacing the set $\Scl$ by the set 
$$S^f=\{ T \subseteq S \mid \textrm{$A_T$ is finite type}\}.$$

The clique-cube complex first appears in a paper of Godelle and Paris  \cite{Go-Pa1}.   
As noted in the introduction,  the Deligne complex is only CAT(0) for a particular type of Artin group, called Artin groups of FC-type. Godelle and Paris prove that the clique-cube complex is CAT(0) for all Artin groups.  

\begin{theorem}[\cite{Go-Pa1}]
The clique-cube complex, $\CG$, is CAT(0) for any graph $\G$. 
\end{theorem}

The group $\AG$  acts on the clique-cube complex $\CG$ by left multiplication, $h\cdot gA_T=(hg)A_T$.  This action preserves the cubical structure and so is an action by isometries. 
The action is also co-compact with fundamental domain consisting of those cubes spanned by $A_\emptyset$ and vertices of the form $A_T$ for some $T\in \Scl.$   However the action is not proper. In particular the stabilizer of a vertex $gA_T$ is the conjugate subgroup $gA_Tg^{-1}$, so all vertices except translates of $A_\emptyset$ have infinite stabilizers. 
We also note that $\CG$ is not a proper metric space since it contains infinite valence vertices.

Each edge in $\CG$ can be labeled with a generator in $S$. For example, the edge between $gA_T$ and $gA_{T\cup \{s\}}$ is labeled $s$.  Moreover, any two parallel edges in a cube have the same label, so we can also label the hyperplane dual to such an edge by $s$.  It is easy to see that every hyperplane of type $s$ is the translate of the hyperplane $H_s$ dual to the edge between $A_\emptyset$ and $A_{\{s\}}$.

The following lemmas illustrate some properties of the clique-cube complex that will be useful in later sections. 
Recall that a simplicial complex is a \emph{flag complex} if whenever a collection of vertices $\{v_0, \dots ,v_k\}$  are pairwise joined by edges, they span a $k$-simplex. In particular, a flag complex is completely determined by its 1-skeleton.  

 \begin{figure}
 
\begin{tikzpicture}[thick, scale=1]


\begin{scope}[shift={(-5.5,5)}]
\draw (150:2) node[]{{\LARGE $\G:$}};

\draw (0,0)--(1,0);
\draw[rotate =150] (0,0)--(1,0)--(1/2,0.87)--(0,0);

\draw (0,0) node[vertex,label={[shift={(0,0.1)}]$c$}]{};
\draw[rotate =150] (1,0) node[vertex,label={[shift={(0,0.1)}]$b$}]{};
\draw (1,0) node[vertex,label={[shift={(0,0.1)}]$d$}]{};
\draw[rotate =210] (1,0) node[vertex,label={[shift={(0,-0.5)}]$a$}]{};

\end{scope}
\begin{scope} [scale=3.5]

\coordinate (E) at (0,0,0);
\coordinate (D) at (1,0,0);
\coordinate (A) at (0,0,1);
\coordinate (C) at (0,1,0);
\coordinate (B) at (-1,0,0);
\coordinate (CD) at (1,1,0);
\coordinate (BC) at (-1,1,0);
\coordinate (AC) at (0,1,1);
\coordinate (AB) at (-1,0,1);
\coordinate (ABC) at (-1,1,1);

\draw (E)-- (D)--(CD)--(C)-- cycle;
\draw (E)--(A)--(AC)-- (C); 
\draw (AC)--(ABC)--(AB)-- (A); 
\draw (C)--(BC)--(ABC); 
\draw [dashed] (E)--(B)--(BC); 
\draw [dashed] (B)--(AB);


\draw (E) node[vertex,label={[shift={(0.3,0.05)}]$A_\emptyset$}]{};
\draw (A) node[vertex,label={[shift={(0.35,-0.6)}]$A_{\{a\}}$}]{};		
\draw (B) node[vertex,label={[shift={(0.45,0.05)}]$A_{\{b\}}$}]{};	
\draw (C) node[vertex,label={[shift={(0.3,0.05)}]$A_{\{c\}}$}]{};	
\draw (D) node[vertex,label={[shift={(0.45,0.05)}]$A_{\{d\}}$}]{};	
\draw (AB) node[vertex,label={[shift={(-0.3,-0.6)}]$A_{\{a,b\}}$}]{};
\draw (AC) node[vertex,label={[shift={(-0.3,0.1)}]$A_{\{a,c\}}$}]{};
\draw (BC) node[vertex,label={[shift={(0.3,0.05)}]$A_{\{b,c\}}$}]{};	
\draw (CD) node[vertex,label={[shift={(0.3,0.05)}]$A_{\{c,d\}}$}]{};		
\draw (ABC) node[vertex,label={[shift={(-0.35,0.1)}]$A_{\{a,b,c\}}$}]{};										


\draw [blue, thick,domain=0:180] plot ({1/3*cos(\x)},{ 1/3*sin(\x)},0);

\filldraw[fill opacity=0.5,fill=blue,draw=blue] (0,1/3,0) arc (90:180:1/3)--(-1/3,0,0) arc (180:245:1/3 and .14)-- (0,0,1/3)arc (203:90: 0.14 and 1/3);

\draw (.28,.27) node[label={[blue,shift={(0,0)}]$lk(A_\emptyset)$}]{};

\end{scope}

\end{tikzpicture}
\caption{The fundamental domain of the complex $\CG$ where $\G$ is the graph shown, with \textcolor{blue}{the link of $A_\emptyset$.}  By Lemma \ref{link iso to Gamma} the flag complex $\tilde{\G}$ is isomorphic to $lk(A_\emptyset)$.  }
\label{fig:Example link}
\end{figure}
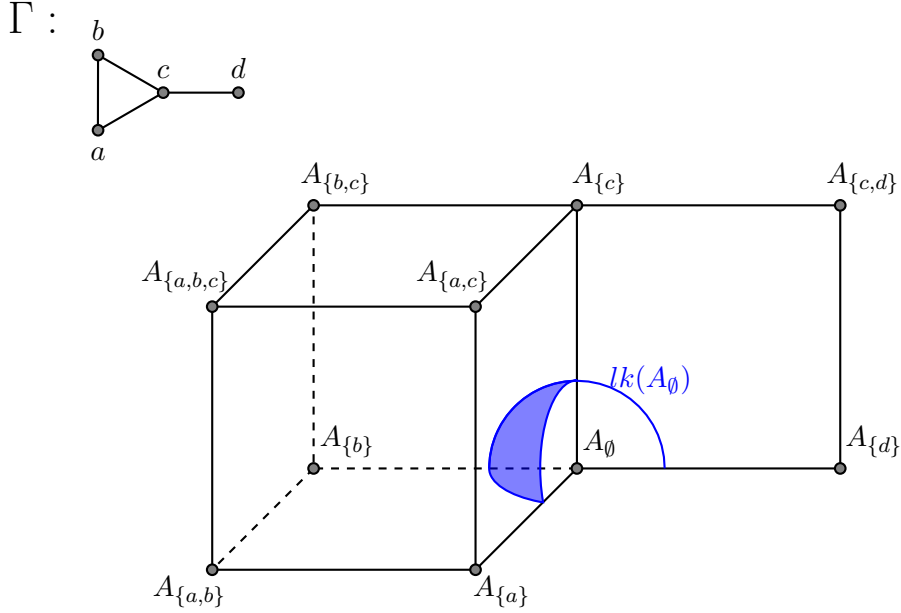 

\begin{lemma}
\label{link iso to Gamma}
In the clique-cube complex, $\CG$, the link of the vertex $A_\emptyset$ is isomorphic to $\tilde{\Gamma}$,  the flag simplicial complex whose 1-skeleton is $\Gamma$. 
\end{lemma}

\begin{proof}
Given two elements of $S$, we know that $s_i$ and $s_j$ are joined by an edge in $\Gamma$ if and only if  $\{s_i,s_j\}\in \Scl$, which happens if and only if $A_\emptyset$ and $A_{\{s_i,s_j\}}$ span a cube.  
The intersection of link of $A_{\emptyset}$ and this this cube is an arc joining  $A_{\{s_i\}}$ and $A_{\{s_j\}}$.   
Thus, $\Gamma$ is isomorphic to the one skeleton of $lk(A_\emptyset)$.  

More generally, $A_\emptyset$ and $A_T$ span a cube in $\CG$ if and only if $T$ spans a clique in $\G$,
that is, if and only if every pair $s_i,s_j \in T$ is joined by an edge.  Thus $k$-simplices in 
$lk(A_\emptyset)$ correspond precisely to the $k$-simplices in $\tilde{\G}$. 
\end{proof}

\begin{lemma}\label{not join implies irreducible}
If the clique-cube complex, $\CG$, decomposes as a product of two subcomplexes, then $\Gamma$ is a join. In particular if $\Gamma$ is not a join,  then $\CG$ is irreducible.  
\end{lemma}

\begin{proof}
If $\CG$ is a product, $\CG=X_1 \times X_2$, then for any vertex $(v_1,v_2)$, the link of $(v_1,v_2)$ is the join of the link of $v_1$ and the link of $v_2$. In particular the link of $A_\emptyset$ must be a join, and therefore  by Lemma \ref{link iso to Gamma}, $\Gamma$ must also be a join. 
\end{proof}

\begin{remark} Here, the condition that $\Gamma$ is a join concerns only the structure of the underlying graph, 
without regard to the labelling.  As a result, the converse of Lemma \ref{not join implies irreducible} is false. For example, suppose $\G$ consists of two vertices $s,t$ joined by an edge and consider the link of the vertex $A_{\{s,t\}}$. The verticies adjacent to $A_{\{s,t\}}$ can be partitioned to into two sets, those of the form $gA_{\{t\}}$ or those of the form $hA_{\{s\}}$. No two verticies of the same form span a cube in $\CG$, so this corresponds to a partition of the verticies of the link into two totally disconnected subcomplexes. Thus, if the link of $A_{\{s,t\}}$ is a join, it must be the join of these two sets.  However, if the label of the edge of $\G$ is an integer greater than 2, then the link is not a join of these two subcomplexes, as for example $A_{\{t\}}$ and $stA_{\{s\}}$ do not span a cube. This shows that the link of $A_{\{s,t\}}$ is not a join and thus, $\CG$ is irreducible.
\end{remark}


To prove our main theorems, we will need to construct geodesic paths in $\CG$ with particular properties.
In a CAT(0) cube complex $X$, we can concatenate geodesic line segments to get a geodesic $\gamma$ if, at each point $p$ where the line segments meet, the distance in $lk(p)$ between the incoming and outgoing segments is at least $\pi$.  

For a vertex $p$ in $X$, the metric on $lk(p)$ is the piecewise spherical metric where each edge has length $\pi/2$, or equivalently, each $k$-simplex is isometric to the intersection of the unit sphere with a quadrant in $\R^{k+1}$.  
An easy exercise shows that for such a simplex $\sigma$, the distance between any two points on disjoint faces
of $\sigma$ is $\pi/2$.  In particular, if two edges in $X$ meet at a vertex, the angle between them is either $\pi/2$ 
(if they span a cube) or $\geq \pi$ (if they do not span a cube). Thus, any edge path in $X$ with the property that no two consecutive edges span a cube, is a geodesic in the CAT(0) metric.

Now consider the link of the vertex $A_\emptyset$ in $\CG$.  As observed above, it is isomorphic to  $\tilde\G$, the flag complex associated to $\G$, with the piecewise spherical metric.
 It follows from the discussion above that that for two vertices $v,w$ in $\G$,
$$d_{\tilde\G}(v,w) = \frac{\pi}{2}\, d_{\G}(v,w)$$
 Moreover, for any two points $x,y \in \tilde\G$,  with $x$ in the simplex spanned by $T_1$ and $y$ in the simplex spanned by $T_2$,
$$d_{\tilde\G}(x,y) \geq \frac{\pi}{2}\, d_{\G}(T_1,T_2).$$

The following lemma will be useful in constructing loxodromic elements of $\AG$ acting on $\CG$. By a loxodromic element we mean one that acts by translation along a bi-infinite geodesic, called the axis of the element.

\begin{lemma}\label{define lox} Suppose that $s_1,\dots ,s_k$ is a sequence of elements of $S$ such that 
$$\textrm{$d_\G(s_i,s_{i+1})> 1$ for all $i$ and $d_\G(s_1,s_k)> 1$.}$$
Set $g=s_1s_2\cdots s_k$ and $g_i=s_1s_2\cdots s_{i}$.  Then $g$ is a loxodromic element with axis 
\[ \dots s_k^{-1}A_\emptyset, \, A_{\{s_k\}},\, A_\emptyset,\,  A_{\{s_1\}},\, g_1A_\emptyset, \, g_1A_{\{s_2\}}, \, g_2A_\emptyset, \, 
g_2A_{\{s_3\}}, \dots, g_{k-1} A_\emptyset, \, g_{k-1}A_{\{s_k\}}, \, g A_\emptyset\dots\]
\end{lemma}

\begin{proof}
Choose a sequence $s_i\in S$ as described in the lemma. The axis described in the lemma is a sequence of edges with labels 
\[\dots s_k, s_k, s_1, s_1, s_2,s_2\dots s_k,s_k\dots\]

Any pair of edges with the matching labels cannot be edges of a single cube. Similarily,  if $s_i$ and $s_{i+1}$ are not joined by an edge in $\G$ then edges in $\CG$ labeled $s_i$ and $s_{i+1}$ cannot be edges of the same cube. Thus this edge path is locally geodesic at each vertex and so forms a geodesic. The element $g$ acts by translation on this bi-infinite axis and so it must be loxodromic. 
\end{proof}

Notice that this lemma shows that if $\Gamma$ is not a clique then $\CG$ has infinite diameter, since if $s_1,s_2$ are not joined by an edge in $\G$, then $g=s_1s_2$ is loxodromic.  Conversely, if $\Gamma$ is a clique, then $\CG$ has finite diameter since every coset $gA_T$ is contained in $A_\G$, hence every vertex of $\CG$ lies in a cube containing the vertex $A_\G$.  Moreover, in this case, the action of $A_\G$ has a global fixed point.  

As shown by Godelle and Paris, the clique-cube complex can be used to reduce many conjectures about general Artin groups to the case of a single clique (see Section \ref{Additional_remarks} below).  However, the clique-cube complex is of little use in the case when $\G$ consists of a single clique.

 \section{Trivial Centers}\label{Trivial Centers}
 
 		 In this section, we will use the clique-cube complex to show that $\AG$ has trivial center for any graph $\G$ that is not a star of a single vertex. In other words, we require the following condition of $\G$: for any vertex $s\in \G$ there exists a vertex $t_s$ such that $d_\G(s,t_s)>1$. 
 
Irreducible, finite-type Artin groups are known to have center isomorphic to $\Z$.  It is conjectured that the only Artin groups with non-trivial center are either of finite type, or are direct products with at least one factor of finite type. If we look at the shape of the graphs defining such Artin groups (ignoring edge labels), we see that the graphs generating such groups are either cliques or joins with one factor a clique. These are exactly the graphs that are stars of a single vertex. 

To prove the full conjecture, however, one would also need to consider the labels on the graphs since not all cliques correspond to finite type Artin groups and not all graph-theoretic joins correspond to direct products of the associated groups.  

The proof  that $\AG$ has trivial center is based on two lemmas.  First, we show that for any graph $\G$, the clique-cube complex $\CG$ cannot be decomposed as a metric product of the form $Y\times \R$.  Second, we show that for any $\G$ that is not a star of a single vertex, the action of $\AG$ on $\CG$ is minimal, that is, it has no proper $\AG$-invariant convex subspace. From these facts, we will deduce that any central element must be trivial.


\begin{lemma}\label{not a product}
For any graph $\G$, the clique-cube complex cannot be decomposed as a metric product of the form $Y\times \R$. 
\end{lemma}

\begin{proof}

Suppose that a CAT(0) cube complex $X$ can be decomposed as the metric product $Y\times \R$ and consider a vertex $v$ in $X$. The link of $v$ is a metric suspension of the set $lk(v,Y) =(Y  \times 0) \cap lk(v)$.    We will show that this suspension has characteristics that are not present in the links of a vertex in $\CG$.  

Let $x_1$ and $x_2$ be the suspension points in $lk(v)$ and suppose these lie in the interior of simplices $\sigma_1$ and $\sigma_2$.  We claim that as a simplicial complex, $lk(v)$ is the join of the subspace spanned by the vertices of $\sigma_1$ and $\sigma_2$ and the subspace spanned by vertices lying in $lk(v,Y)$.  In particular, any vertex in $\sigma_1$ or $\sigma_2$ is connected by an edge to all but finitely many vertices of $lk(v)$.  

First note that in any piecewise spherical flag complex with all edges of length $\pi/2$, the only vertices at distance $\pi /2$ from a point in the interior of a simplex $\sigma$ are vertices $w$ such that $\sigma \cup \{w\}$ spans a simplex.  Now suppose that $y$ is a vertex of $lk(v)$ which is not in either $\sigma_1$ or $\sigma_2$. 
Then $d_{lk(v)}(y,x_i)\geq \pi/2$ for  $i=1,2$.  Since $lk(v)$ is a suspension, this is possible only if $y \in lk(v,Y)$ and $d_{lk(v)}(y,x_i)=\pi/2$.  Hence $y$ must span a simplex with $\sigma_i$, which proves the claim.

We will now describe a vertex in $\CG$ which does not have this property. That is we would like to find a vertex $v\in\CG$ such that any vertex in $lk(v)$ is not connected by an edge to infinitely many vertices in $lk(v)$.

 Let $T$ be a maximal clique in $\G$, and consider the link of the vertex $A_T$. Any vertex in the link of $A_T$ corresponds to an edge in $\CG$ labeled with a generator $t\in T$. Thus the vertices in $lk(A_T)$ can be partitioned into subsets corresponding to the elements of $T$. Any two distinct vertices $gA_{T\smallsetminus\{t\}}$ and $hA_{T\smallsetminus\{t\}}$ in the same subset do not span a cube in $\CG$ and so cannot be connected by an edge in the link. Moreover, for any $t\in T$ the subset of vertices in the link corresponding to edges labled $t$ contains infinitely many distinct elements, one corresponding to each vertex $t^nA_{T\smallsetminus\{t\}}$, for any $n\in \Z$.   Thus every vertex in $lk(A_T)$ is not connected to infinitely many other vertices, and so $lk(A_T)$ cannot be a metric suspension and $\CG$ cannot be a metric product of the form $Y\times \R$.
\end{proof}

\begin{lemma}\label{minimal action} Suppose that $\G$ is a graph that is not a star of a single vertex. The action of $\AG$ on $\CG$ is minimal.  That is, for any point $x  \in \CG$, the convex hull of the orbit of $x$ is all of $\CG$.   
\end{lemma}

\begin{proof}  First we prove this lemma in the case where $x$ is a vertex of the cube complex. Denote the convex hull of the orbit of $x$ by Hull($orb(x)$).  We may assume that $x$ is of the form $A_T$ for some clique $T$ in $\G$. We will show that Hull$(orb(A_\emptyset))\subset$ Hull$(orb(A_T))$. 

Fix $t\in T$ and, using the fact that $\G$ is not the star of a single vertex, choose $s$ not in the link of $t$. Let $Q=T\cap lk(s)$ and consider the path
\[A_T, A_Q,A_{Q\cup \{s\}},sA_Q,sA_T\]

This path is geodesic. Between $A_T$ and $A_Q$ the path crosses only hyperplanes labeled by generators is $T\smallsetminus Q$, none of which are in $lk(s)$. Between $A_Q$ and $A_{Q\cup \{s\}}$ only crosses the hyperplane labeled by s, so the above path is locally geodesic at $A_Q$. Similarly, it is locally geodesic at $A_{Q\cup \{s\}}$ as  approaches and leaves this vertex by crossing hyperplanes labeled by $s.$ The part of the path passing through $sA_Q$ is a translation of the path through $A_Q$ and so the path is locally geodesic at this vertex as well. 

This shows that $A_Q$ lies in the convex hull of the orbit of $A_T$ and hence  Hull$(orb( A_Q))\subset$ Hull$(orb( A_T))$.  The fact that $\G$ is not the star of a single vertex guarantees that $Q\subsetneq T$. This process can be repeated for successively smaller cliques to show that  Hull$(orb( A_\emptyset))\subset$ Hull$(orb( A_T))$.  

Next, we choose an arbitrary $R\in \Scl$ and show that Hull$(orb( A_R))\subset$ Hull$(orb( A_\emptyset))$ by induction on $|R|$. The two edges joining $A_{R\smallsetminus\{r\}}$ to $A_R$ to $rA_{R\smallsetminus\{r\}}$ are geodesic so 
Hull$(orb( A_R))\subseteq$
Hull$(orb( A_{R\smallsetminus\{r\}}))$. 
By induction this means that Hull$(orb( A_R)) \subseteq$ Hull$(orb( A_\emptyset))$.
Putting together the above arguments we see that  Hull$(orb( A_T))$ contains every vertex in $\CG$ and so must be the whole cube complex $\CG$. Thus for any clique $T$ and any vertex $A_T$, Hull$(orb( A_T))=\CG$

Now we generalize to the case where $x$ lies in the interior of a $k$-cube $C$. Translating by an element of $\AG$ if necessary, we may assume $C$ is spanned by vertices $A_T$ and $A_R$, with $T \subseteq R$.  Let $k = |R - T|$.  We proceed by induction on $k$.

For $k=0$, $x=C$ is a vertex, so this follows from the above argument.

Say $k \geq 1$.  Let $s \in R-T$.  Then the cubes $C$ and $sC$ share a codimension one face
(spanned by $A_s$ and $A_R$).  The geodesic from $x$ to $sx$ clearly passes through this face,
say at the point $y$.  Then $y$ lies in Hull$(orb( x))$, so Hull$(orb( y)) \subseteq$ Hull$(orb( x))$. By induction, Hull$(orb( y)) = \CG$.
\end{proof}


Now we bring together the above results to show that the center of the group is trivial
\begin{theorem}\label{thm: trivial center}  Assume that $\G$ is not the star of a single vertex.  Then the center of $\AG$ is trivial.
\end{theorem}

\begin{proof}  Suppose for contradiction that $z$ is a central element.  We first observe that the minset of $z$ is all of $\CG$.  Recall that 
\[min(z):=\{x\in \CG \mid d(x,zx)=
\inf_{y\in \CG}
d(y,zy)\}\]

 To see that $min(z)=\CG$, note that if $x$ is in the minset, then so is every point in the orbit of $x$, since $d(x,zx)=d(gx,gzx)=d(gx,zgx)$.  The minset is convex, so it follows from Lemma \ref{minimal action} that it contains all of $\CG$.  Such an element is called a Clifford Translation (\cite{Br-Ha}, II.6.8).

Now we consider two cases. Any simplicial isometry on a CAT(0) cube complex must either be elliptic, with a fixed point, or loxodromic (\cite{Br-Ha}, II.6.6).   
 
 If $z$ is loxodromic, then the minset of $z$ decomposes as a metric product,
$\CG=min(z)= \R \times Y$. However by Lemma \ref{not a product}, $\CG$ cannot decompose in this way, providing a contradiction. 

If $z$ is elliptic, it fixes a point, hence it must fix all of $\CG$.  In particular, it fixes the vertex $A_\emptyset$.  Since $A_\emptyset$ has trivial stabilizer, this implies that $z =1$.
\end{proof}

 \section{Acylindrical Hyperbolicity}
 		
 A hyperbolic group is a group that acts properly, cocompactly by isometries on a hyperbolic metric space. 
These groups were introduced by Gromov in the 1980s and have been extensively studied.  More recently, there has been an interest in extending some of these ideas and techniques to more general groups.   One such class of groups are acylindrically hyperbolic groups.  
 
\begin{definition}
The action of a group $G$ on a metric space $X$ is called \emph{acylindrical} if for every $\epsilon > 0$ there exist 
$R, N > 0$ such that for every two points $x, y$ with $d(x, y) \geq R$, there are at most $N$ elements $g \in G$ satisfying $d(x,gx) \leq \epsilon$ and $d(y,gy) \leq \epsilon$.  We say $G$ is \emph{acylindrically hyperbolic} if it is not virtually cyclic and has an acylindrical action  with unbounded orbits on a hyperbolic metric space.
\end{definition}

For an in depth discussion of this property, its consequences, and the types of groups it encompasses, we recommend Osin's  article \cite{Os}. For a discussion of acylindrical hyperbolicity in the context of CAT(0) cube complexes, we recommend the paper of Genevois \cite{Gen}.  In this section, we discuss acylindrical hyperbolicity for Artin groups.
 
In \cite{Ch-Ma}, Chatterji and Martin establish criteria for a groups acting on a CAT(0) cube complex to be acylindrically hyperbolic.  
Recall that the action of a group $G$ on a CAT(0) cube complex $X$ is \emph{essential} if no orbit remains at bounded distance from a half-space of $X$.  The action is \emph{non-elementary} if it does not admit a finite orbit in $X \cup \partial X$, where $\partial X$ is the visual boundary of $X$. 

\begin{theorem}[Chatterji and Martin]  \label{Chatterji and Martin}
 Let G be a group acting essentially and non-elementarily on an irreducible finite-dimensional CAT(0) cube complex $X$. If there exist two hyperplanes in $X$ whose stabilizers intersect in a finite subgroup, then G is acylindrically hyperbolic.
\end{theorem}

We remark that the theorem does not require that $X$ be hyperbolic and does not imply that the action of $G$ on $X$ is acylindrical.  Rather it uses (a variant of) the Bestvina-Bromberg-Fujiwara's  ``WPD condition" \cite{BeBrFu, Mar} to show that $G$ acts acylindrically on \emph{some} hyperbolic space.  

Under the assumption that the defining graph $\G$ has diameter at least 3, Chatterji and Martin show that the action of an Artin group $\AG$ of FC-type on its Deligne complex satisfies the conditions of their theorem, and hence $\AG$ is acylindrically hyperbolic.  We strengthen their results by getting rid of the FC-type condition, and only assuming that $\G$ is not a join. To do this, we will verify that the action of $\AG$ on the clique-cube complex satisfies the hypotheses of their theorem.  

By Corollary \ref{not join implies irreducible}, $\CG$ is irreducible. We begin by showing that the action is essential and non-elementary.

\begin{lemma}\label{non-elementary action}
If $\G$ is not a join and has at least two vertices, then the action of $\AG$ on $\CG$ is essential and non-elementary. 
\end{lemma}

\begin{proof}
Let $\G^c$ be the complement graph to $\G$, that is the graph with vertex set $S$ and an edge between two vertices $s_1$ and $s_2$ if and only if there is no edge between $s_1$ and $s_2$ in $\G$. 
The complement $\G^c$ is a connected graph because $\G$ is not a join, so we can find a sequence $s_1,\dots, s_k\in S$ such that $d_\G(s_i,s_{i+1})> 1$, $d_\G(s_1,s_k)> 1$.  Moreover, we can choose the sequence so that it passes through every vertex of $\G$ at least once, and, as $\G$ is not a join, $\cap_i lk(s_i)=\emptyset$. 
By Lemma \ref{define lox}, $g=s_1\cdots s_k$ is a loxodromic isometry whose axis $\ell_g$ is an edge path.  

 Now we establish some properties of this element $g$. For an arbitrary generator $s\in S$, let $H_s$ denote the hyperplane intersecting the edge between $A_\emptyset$ and $A_{\{s\}}$.  Since the cubical neighborhood of $A_\emptyset$ is a fundamental domain for the action of $\AG$, every hyperplane in $\AG$ is a translate of some $H_s$.  We will say that such a hyperplane is of type $s$. Notice that a hyperplane of type $s$ only crosses hyperplanes of type $t$ for $t\in lk(s)$. In particular, if we consider the hyperplanes that cross the edge path $\ell_g$, we see that any hyperplane of type $s_i$ cannot cross any hyperplane of type $s_{i+1}$. This shows that  the edge path $\ell_g$ has the following properties.
\begin{enumerate}
\item No two hyperplanes crossed by $\ell_g$ intersect each other.
\item For any point $x \in \ell_g$, the segment of $\ell_g$ from $x$ to $gx$ crosses a hyperplane of type $s$ for every vertex $s$ of $\G$.  In particular, no hyperplane $H$ of $X$ can intersect all of the hyperplanes crossed by $[x,gx]$. 
\end{enumerate}
It follows from the last property, that $\ell_g$ is not parallel to any hyperplane in $X$ and hence cannot bound a half flat.  That is, $g$ is rank one.

We now show that the action is essential.  The action is cocompact, so it suffices to show that each hyperplane is essential, that is that both halfspaces of any hyperplane contain points arbitrarily far from the hyperplane.  By property (1), this is clear for any hyperplane that crosses $\ell_g$.  By property (2), every hyperplane $H$ is a translate of some hyperplane that crosses $\ell_g$, so the same holds for every $H$.

Next we show that the action is non-elementary.
The existence of loxodromic elements means that there are no points in $\CG$ with finite orbit. It remains to show there are no points with finite orbit in $\partial \CG$. 

It is well known for actions on proper CAT(0) spaces, that a rank one element fixes only two points on the boundary, namely the endpoints of its axis.  While $\CG$ is not proper, a similar argument applies.  For suppose there was another fixed point, $z$, on the boundary, and let $\gamma$ be a geodesic from z to $\ell_g$, intersecting the axis at only one point. If $z=gz,$ then $\gamma$ and $g\gamma$ are asymptotic, and by the Flat Strip theorem (see \cite{Br-Ha}, II.2.13) the convex hull of $\gamma$ and $g\gamma$ embedded into $\R^2$. Putting together the flat strips corresponding to $g^n\gamma$, we see that in this case $\ell_g$ must bound a half flat. 

Now set $h=s_1g=(s_1)^2s_2 \dots s_k$.  By the exactly the same argument as for $g$, one can show that $h$ is a rank one loxodromic element with axis given by an edge path $\ell_h$, and its only fixed points on $\partial \CG$ are the endpoints of this axis.  

Suppose there exists a finite orbit in $\partial \CG$.  Then some power of $g$ and some power of $h$ must fix the same point.  But powers of $g$ can only fix the endpoints of $\ell_g$ and powers of $h$ can only fix endpoints of $\ell_h$, so it suffices to show that these endpoints are all distinct. 

Let $x=A_\emptyset$ and $y=A_{\{s_1\}}$.  Note that $g^{-1}y=h^{-1}y$.  The axes $\ell_g$ and $\ell_h$ intersect along the segment from $[g^{-1}y, y]$.  
Let $\ell_g^+$ and $\ell_h^+$ denote the positive rays emanating from $y$.  These two rays begin with the edges $[y, s_1x]$ and $[y, s_1^2x]$, respectively.  Since these two edges do not span a cube, the two rays $\ell_g^+$ and $\ell_h^+$ combine to form a bi-infinite geodesic, and hence they must have distinct endpoints.  
Similarly, the negative rays $\ell_g^-$ and $\ell_h^-$ emanating from $g^{-1}y$ ($=h^{-1}y$) begin with 
$[g^{-1}y,g^{-1}x]$ and $[h^{-1}y,h^{-1}x]$ respectively, and combine to form a bi-infinite geodesic (see Figure \ref{fig:Axes1}).  It follows that the four endpoints of $\ell_g$ and $\ell_h$ are all distinct.  

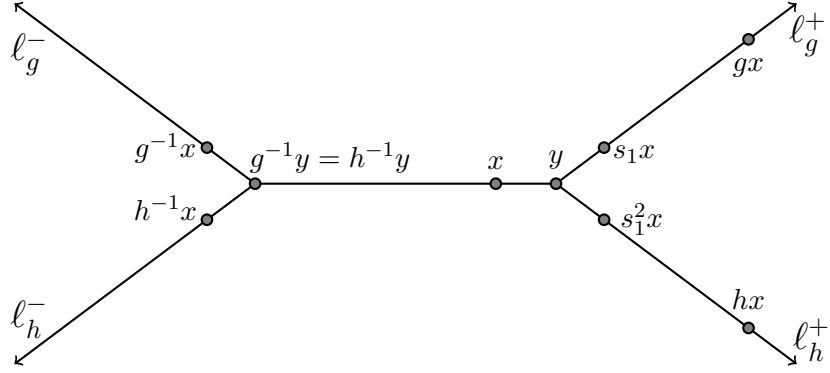
\begin{figure}
 
\begin{tikzpicture}[thick, scale=.8]

  						 \draw (0,0)--(5,0);
    								
    					\draw[->]	 (0,0)--(-4,3);
    					\draw[->] (0,0)--(-4,-3);
    					\draw[->] (5,0)--(9,3);
    					\draw[->] (5,0)--(9,-3);


		\draw (0,0) node[vertex,label={[shift={(1,0.05)}]$g^{-1}y=h^{-1}y$}]{};	
		\draw (5,0) node[vertex,label={[shift={(0,0.1)}]$y$}]{};		
		\draw (4,0) node[vertex,label={[shift={(0,0.1)}]$x$}]{};	
		
		\draw (29/5,3/5) node[vertex,label={[shift={(0.4,-0.3)}]$s_1 x$}]{};
		\draw (29/5,-3/5) node[vertex,label={[shift={(0.5,-0.3)}]$s_1^2 x$}]{};	
		\draw (-4/5,-3/5) node[vertex,label={[shift={(-.55,-0.1)}]$h^{-1} x$}]{};
	    \draw (-4/5,3/5) node[vertex,label={[shift={(-0.55,-0.3)}]$g^{-1} x$}]{};
	    
	    \draw (41/5,12/5) node[vertex,label={[shift={(0,-0.6)}]$g x$}]{};	
	    \draw (41/5,-12/5) node[vertex,label={[shift={(0,.15)}]$h x$}]{};

		\draw (9,3) 	node[label={[shift={(0.15,-0.75)}]{{\Large $\ell_g^+$}}}]{};	 
		\draw (9,-3) 	node[label={[shift={(0.2,0)}]{{\Large $\ell_h^+$}}}]{};
		\draw (-4,-3) 	node[label={[shift={(0.2,0.3)}]{{\Large $\ell_h^-$}}}]{};	
		\draw (-4,3) 	node[label={[shift={(0.2,-1)}]{{\Large $\ell_g^-$}}}]{};	   				
\end{tikzpicture}
\caption{Axes of the loxodromic elements $g$  and $h$ as defined in the proof of Lemma \ref{non-elementary action}, where $x=A_\emptyset$ and $y=A_{\{s_1\}}$}
\label{fig:Axes1}
\end{figure}

This proves that  there is no finite orbit in  $\CG \cup \, \partial \CG$ and the action of $\AG$ is non-elementary.
\end{proof}

\begin{theorem}\label{thm: acyl hyp}
Any Artin group $\AG$ whose defining graph $\G$ is not a join and has at least two vertices is acylindrically hyperbolic.  
\end{theorem}

\begin{proof} It remains only to show that there exist a pair of hyperplanes whose stabilizers have finite intersection.  We can then apply Theorem \ref{Chatterji and Martin} to obtain the desired result. 

Let $g=s_1, \dots ,s_k$ be the group element constructed in Lemma \ref{non-elementary action}.  This element was designed so that its axis, $\ell_g$, is a geodesic edge path which crosses a sequence of hyperplanes such that no two consecutive hyperplanes intersect, and no hyperplane in $\CG$ intersects all of the hyperplanes crossed by a segment $[x,gx]$ of $\ell_g$.  

Take $x=A_\emptyset$, a point on $\ell_g$. Chose any two hyperplanes $H$ and $H'$ which cross $\ell_g$ on opposite sides of the segment $[x,gx]$.  Let $N(H)$ denote the cubical neighborhood of $H$, that is, the union of cubes intersecting $H$.  Let $\alpha$ be the segment of $\ell_g$ connecting $N(H)$ to $N(H')$.  
We claim that $\alpha$ is the unique minimal length edge path connecting these two neighborhoods.  It will follow from this claim that if $a \in \AG$ stabilizes both $H$ and $H'$, then $a$ must fix $\alpha$.  In particular, it fixes the vertex $A_\emptyset$, so $a=1$.

To verify the claim, let $\mathcal H_\alpha$ denote the set of hyperplanes crossed by $\alpha$.  Since $\alpha$ contains the segment $[x,gx]$, $\mathcal H_\alpha$ contains hyperplanes of every possible type.  
Since no two hyperplanes that cross $\ell_g$ can intersect, every hyperplane in $\mathcal H_\alpha$  separates  $H$ from $H'$, hence any edge path between these two hyperplanes must cross every hyperplane in $\mathcal H_\alpha$.  Moreover, none of these hyperplanes cross each other, so $\alpha$ is the unique minimal length edge path connecting its own endpoints.  

Now suppose we start at some other vertex $y$ in  $N(H)$.  Then $y$ and the initial point of $\alpha$ are separated by some hyperplane $H''$ which intersects $H$.   $H''$ cannot cross $\alpha$ since no two hyperplanes which cross $\ell_g$ intersect each other.  Nor can it cross every hyperplane in $\mathcal H_\alpha$, since two hyperplanes of the same type cannot intersect. Thus, $H''$ must separate $y$ from $H'$.  It follows that any edge path from $y$ to $H'$ crosses at least one hyperplane in addition to $\mathcal H_\alpha$.  This proves the claim and completes the proof of the theorem.

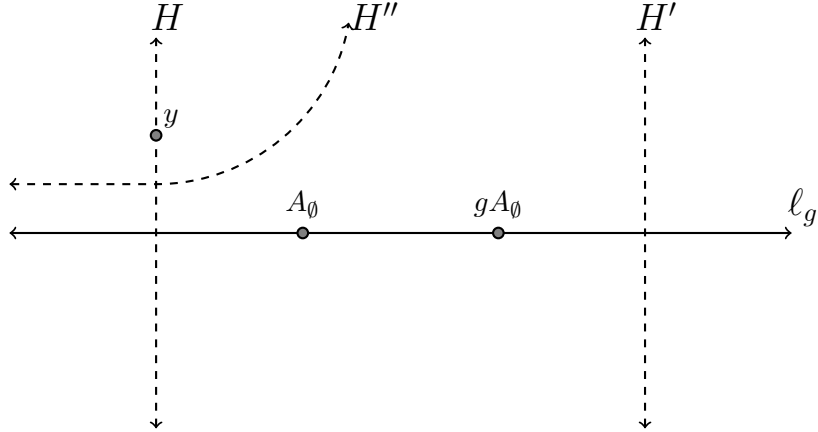
\begin{figure}
\begin{tikzpicture}[thick, scale=1.3]


\draw[<->] (-3,0)--(5,0);

\draw (0,0) node[vertex,label={[shift={(0,0.1)}]$A_\emptyset$}]{};

\draw (2,0) node[vertex,label={[shift={(0,0.1)}]$gA_\emptyset$}]{};

\draw (5,0) 	node[label={[shift={(0.15,0.05)}]{{\Large $\ell_g$}}}]{};	 


\draw[dashed, <->] (3.5,-2)--(3.5,2);

\draw (3.5,2) 	node[label={[shift={(0.15,0.05)}]{{\Large $H'$}}}]{};	

\draw[dashed, <->] (-1.5,-2)--(-1.5,2);

\draw (-1.5,2) 	node[label={[shift={(0.15,0.05)}]{{\Large $H$}}}]{};	

\draw (-1.5,1) node[vertex,label={[shift={(.2,0)}]$y$}]{};

\draw[dashed, <->] (-3,0.5)--(-1.5,0.5) arc (-90:-10:2);

\draw (0.5,2) 	node[label={[shift={(0.3,0.05)}]{{\Large $H''$}}}]{};

\end{tikzpicture}
\caption{The axis of $\ell_g$ is the shortest edge path between the cubical hyperplane neighborhoods $N(H)$ and $N(H')$, as the hyperplane $H''$ cannot follow parallel to $\ell_g$ and cannot intersect $\ell_g$ between $A_\emptyset$ and $gA_\emptyset$. }
\label{fig:disjoint_stabilizer_proof}
\end{figure}
\end{proof}

\begin{remark} Rather than using the theorem of Chatterji and Martin, one could instead use a theorem of  Genevois (Theorem 17, \cite{Gen2}) to conclude that the element $g$ constructed in the proof above is a contracting, WPD element.  It then follows directly from either Bestvina-Bromberg-Fujiwara \cite{BeBrFu} or Sisto \cite{Sis} that $\AG$ is acylindrically hyperbolic.
\end{remark}

\section{Remarks on Other Conjectures} \label{Additional_remarks}

In \cite{Go-Pa1} and \cite{Go-Pa2}, Godelle and Paris show that several of the conjectures listed in the introduction can be reduced to the case where the defining graph is a single clique.  In this section, we observe that the clique-cube complex gives particularly simple proofs of this reduction for two of these conjectures.

The easiest of these is the conjecture that $\AG$ is torsion-free. 

\begin{theorem}  If $A_T$ is torsion-free for every clique $T$ in $\G$, then $\AG$ is torsion-free.
\end{theorem}

\begin{proof}  Suppose $g \in \AG$ has finite order.  Then $g$ fixes a point $x$ in $\CG$.  Since the action of $g$ preserves the type of every hyperplane, $g$ must fix the entire cube containing $x$.  In particular, it fixes some vertex $aA_T$.  Thus $g$ lies in $aA_Ta^{-1}$ which, by assumption, is torsion free.
\end{proof}

Recall the the $K(\pi,1)$-conjecture for $\AG$ states that, after restricting the hyperplane complement $\mathcal H_\G$ to an open cone $\mathcal V_\G$ (the Vindberg cone),  the quotient space $X_\G = \mathcal V_\G / \WG$, is a $K(\pi,1)$-space for $\AG$.  This space is known to have fundamental group $\AG$, so the conjecture reduces to showing that the universal covering space of $X_\G$ is contractible.  While the space $X_\G$ is not compact, in \cite{Ch-Da2}, the authors show that it is homotopy equivalent to a finite complex, known as the Salvetti complex, for any graph $\G$. Thus, if the $K(\pi,1)$-conjecture holds, then $\AG$ also has a finite $K(\pi,1)$-space.  

Using the clique cube complex, we can reduce the $K(\pi,1)$-conjecture, and hence the existence of a finite $K(\pi,1)$-space, to the case of a single clique.  This was first proved by  Ellis-Sk\"oldberg in \cite{El-Sk} and Godelle-Paris in \cite{Go-Pa1}.   

\begin{theorem}  If $A_T$ satisfies the $K(\pi,1)$-conjecture for every clique $T$ in $\G$, then $\AG$ satisfies the 
$K(\pi,1)$-conjecture.
\end{theorem}

\begin{proof}  In \cite{Ch-Da1}, it is shown that the universal covering space of $X_\G$ is homotopy equivalent to the Deligne complex $\DG$, thus  the theorem can be restated as follows:  if $\DT$ is contractible for every clique $T$ in $\G$, then $\DG$ is contractible.

We can view the Deligne complex $\DG$ as the subcomplex of $\CG$ consisting of all cubes whose vertices are cosets of the form $aA_T$ with $A_T$ of finite type.  Now let $v=aA_T$ be any vertex in $\CG$.  Define the ``downward link" of $v$, $dl(v)$, to be the subcomplex of $\CG$ spanned by the set  of all vertices $bA_R$ with 
 $bA_R \subsetneq aA_T$.  
 
 We first prove by induction on $|T|$, that if $v=A_T$ is of infinite type, then $dl(v)$ is contractible (and hence likewise for $v=aA_T$).   Every 2-clique in $\G$ generates a finite type Artin group, so we begin with $|T|$=3.  In this case,  $dl(v)=\DT$ which is contractible by hypothesis.  Now say $|T|$=k.  Once again, $\DT$ is contractible by hypothesis,   and to obtain $dl(v)$ from $\DT$, we must add the vertices of the form $w=bA_R \subsetneq A_T$ where $R$ is a clique and $A_R$ is of infinite type.  We do this inductively on $|R|$, so that at each stage, all smaller vertices have already been attached.  Thus, topologically, attaching $w$ has the effect of coning off $dl(w)$.  Since $|R| < |T|$, we know by induction that $dl(w)$ is contractible, hence coning it off does not change the homotopy type.  Since $dl(v)$ is obtained from $\DT$ by attaching a series of these cones, we conclude that $dl(v)$ is also contractible.  
 
 Now consider the full graph $\G$.  By the same argument, $\CG$ is obtained from $\DG$ by inductively coning off 
 $dl(v)$ for each vertex $v=aA_T$ with $A_T$ of infinite type.  By the previous paragraph, these downward links are all contractible, hence $\DG$ is homotopy equivalent to $\CG$.  Since $\CG$ is CAT(0), it is contractible, so the same holds for $\DG$.
 \end{proof}

 As noted in Section \ref{Cube Complex}, the clique-cube complex is not a useful tool when $\G$ is a single clique since in this case $\CG$ has finite diameter and the action of $\AG$ has a global fixed point.  So the next goal is to find a useful complex to study the single-clique case.

We remark that recent work of Juh\'asz \cite{Juh}, using more combinatorial methods, gives some other reductions of these conjectures.   In particular, he reduces the $K(\pi,1)$-conjecture and the torsion-free conjecture to the case
of a graph $\G$ where every vertex is contained in some finite type subgroup $A_T$ with $|T| \geq 3$.  He also states that in a forthcoming paper he will use similar methods to prove the trivial center conjecture for graphs which do \emph{not} satisfy this property. 	
		
\bibliographystyle{plain}
\bibliography{full_biblio}		

\end{document}